\documentclass[11pt,twoside]{article}
\usepackage{latexsym}
\usepackage{amssymb,amsbsy,amsmath,amsfonts,amssymb,amscd,amsthm}
\usepackage{epsfig, graphicx}
\usepackage{color}
\usepackage{enumerate}
\newcommand{\R}{\mathbb{R}}
\newcommand{\N}{\mathbb{N}}
\newcommand{\Z}{\mathbb{Z}}

\newcommand {\e}  {\varepsilon}

\newcommand\esup{\mathop{\mathrm{ess\,sup}}}
\newcommand\einf{\mathop{\mathrm{ess\,inf}}}

\newcommand {\cB}  {{\cal B}}

\newcommand {\cM}  {{\cal M}}

\newcommand {\J}{{\mathcal J}}
\newcommand{\cX}{{\mathcal X}}

\newcommand {\f}   {\frac}

\def\r{\rho}
\newcommand{\dis}{\displaystyle}

\newcommand{\beq}{\begin{equation}}
\newcommand{\beqa}{\begin{eqnarray}}
\newcommand{\bea} {\begin{array}{ll}}
\newcommand{\beqan}{\begin{eqnarray*}}
\newcommand{\eeq}{\end{equation}}
\newcommand{\eeqa}{\end{eqnarray}}
\newcommand{\eeqan}{\end{eqnarray*}}
\newcommand{\eea} {\end{array}}
\newtheorem{theorem}{Theorem}[section]
\newtheorem{lemma}[theorem]{Lemma}
\newtheorem{definition}[theorem]{Definition}

\newtheorem{corollary}[theorem]{Corollary}
\newtheoremstyle{remarkb}
	{}
	{}
	{\normalfont}
	{}
	{\bfseries}
	{}
	{ }
	{}
\theoremstyle{remarkb}
\newtheorem{remark}[theorem]{Remark}
\numberwithin{equation}{section}
\newcommand{\cqfd}{{ \hfill
                       {\unskip\kern 6pt\penalty 500
                       \raise -2pt\hbox{\vrule\vbox to 6pt{\hrule width 6pt
                       \vfill\hrule}\vrule} \par}   }}
\title{\Large \bf On numerical approximation  of the Hamilton-Jacobi-transport system arising in high frequency approximations}

\author{Yves Achdou~$^{\text{a}}$, Fabio Camilli~$^{\text{b}}$,  Lucilla Corrias~$^{\text{c}}$
}
\date{\today}
\begin{document}
\maketitle
\pagestyle{plain}
\pagenumbering{arabic}

\begin{abstract}
In the present article, we study the numerical approximation of a system of Hamilton-Jacobi and transport
equations arising in geometrical optics. We consider a semi-Lagrangian scheme.
We prove the well posedness of the discrete problem and the convergence of the approximated solution toward the viscosity-measure valued solution of the exact problem.
\end{abstract}



\section{Introduction}
\label{Sec:intro}
In this article, we consider the following   system
\begin{align}
&\partial_tu +H(x,t, \nabla u) = 0\,,  &\hbox{in } \R^d\times (0,T],\label{HJ}\\
&\partial_tm+\nabla\cdot(a(x,\nabla u)\,m)  = 0\,, &\hbox{in } \R^d\times (0,T],\label{CE}\\
&u(x,0)=u_0(x),\quad m(0)=m_0\,, &\hbox{in } \R^d\,,\label{IC}
\end{align}
of an Hamilton-Jacobi type equation (HJ) and a continuity equation (CE) describing the transport of the conserved measure $m_0$. Even if the vector field $a$ can be smooth (in the simplest and of reference case, $a(x,p)=p$),  the scalar function $u$ as a solution of \eqref{HJ} is intended in the viscosity sense. Therefore, $\nabla u$ is at most~$L^\infty$, even for smooth initial data $u_0$, and the regularity that we can expect for the vector field $a(\cdot,\nabla u)$ is the very weak regularity
\begin{equation}
a(\cdot\,,\nabla u(\cdot,\cdot))\in (L^\infty(\R^d\times[0,T]))^d\,.
\label{weakhyp}
\end{equation}
As a consequence, despite the fact that the (HJ) equation can be solved independently from (CE), system \eqref{HJ}-\eqref{CE} leads to  some interesting mathematical issues as well as numerical challenges.

More specifically, in order to obtain global solutions, a well adapted notion of solution for \eqref{CE} under hypothesis \eqref{weakhyp} is the notion of {\sl measure solutions} introduced by Poupaud-Rascle \cite{PR}. Indeed, the latter perfectly makes sense if the vector field $a(\cdot,\nabla u)$ satisfies a {\sl one-sided Lipschitz condition} (OSL), (see \eqref{OSL} below), and this condition can be obtained by the semiconcavity of the viscosity solution of \eqref{HJ}, at least in the reference case $a(x,p)=p$, or in the one dimensional case for a class of functions $a$.

Existence and uniqueness results for problem \eqref{HJ}-\eqref{CE}-\eqref{IC} in the framework of viscosity-measure solutions have been given for example in \cite{BK, St} and in the different framework of the one dimensional viscosity-duality solutions in \cite{GJ}. Therefore, our goal here is not to refine these previous results but to construct consistent numerical approximations.

The considered scheme is based on a semi-lagrangian discretization of \eqref{HJ}-\eqref{CE} for the time approximation coupled with a finite element discretization for the space variable. It requires a convex hamiltonian $H$.
 Since the semiconcavity of the initial data $u_0$ is conserved by the scheme, it is sufficient to require only a weak  (OSL) condition at the discrete level
 (automatically satisfied in the reference case $a(x,p)=p$) without semiconcavity requirement for the viscosity solution $u$.


The stability of the Filippov characteristics and of the corresponding measure solution of the transport equation~(CE) will be the key tool to prove the convergence of the numerical schemes. As a by-product, we obtain of course a new existence and uniqueness result for the viscosity-measure solution of \eqref{HJ}-\eqref{CE}-\eqref{IC}.

Regarding the applications of our numerical approximations of \eqref{HJ}-\eqref{CE}-\eqref{IC}, it is worth to recall that these types of systems arise for example in the semi-classical limit for the Schr{\"o}dinger equation \cite{Ca} and of the spinless Bethe-Salpeter equation (the relativistic Schr{\"o}dinger equation, \cite{BS}), or in  the high frequency approximation of the Helmholtz equation. In these cases, the hamiltonian $H$ and the vector field $a$ take the forms
\begin{align}
&H(x,t,p)=\f12|p|^2+V(x,t)\,,\quad a(x,p)=p\,,\quad (\hbox{\sl Schr{\"o}dinger\ and\ Helmholtz\ equation}),
\label{SH}\\
&H(x,t,p)=\left(\f{|p|^2}2+1\right)^{1/2}+V(x,t)\,,\quad a(x,p)=p\!\left(\f{|p|^2}2+1\right)^{-1/2}\!\!\!\!\!,\quad (\hbox{\sl Bethe-Salpeter\ equation}),
\label{BS}
\end{align}
where $V(x,t)$ is the potential (see \cite{BK,GJ,St} for a rapid derivation of \eqref{SH} and \eqref {BS}).\par
Coupling between first  or second  order Hamilton-Jacobi equations and transport equations has been also recently considered  in the framework of Mean Field Games theory \cite{ll}. In this case however, the system is given by a backward Hamilton-Jacobi equation and a forward transport equation with initial-terminal
conditions and   coupling terms in both  equations. The resulting system turns out to be completely different from the one  considered here.

The paper is organized as follows. Section \ref{Sec:preliminaries} is devoted to the preliminary definitions and known results concerning system \eqref{HJ}--\eqref{CE}--\eqref{IC}. In Section \ref{Sec:SLscheme} we construct a semi-lagrangian scheme for the time discretization of system \eqref{HJ}--\eqref{CE}--\eqref{IC} and we prove its convergence.  The corresponding fully discrete  scheme is given and analyzed in Section \ref{Sec:fully discrete}. 
 Finally, the appendix is devoted to the proof of some technical lemmas for the reader convenience.

Throughout the paper, we will denote by $C^0_0(\R^d)$ (resp. $C^0_c(\R^d)$) the space of continuous functions which tend to 0 at infinity (resp. with compact support); by  $\rho_\e$ a standard mollifier, i.e. $\r_\e(x)~\!=\!\frac{1}{\e^d}\rho(\frac{x}{\e})$, $\r\in C_c^\infty(\R^d)$, $\r\ge0$ and $\int_{\R^d}\r(x)dx=~1$; by $*$ the convolution with respect the space variable and by $C$ any numerical constant that can vary from line to line in the computations.

\section{Preliminaries : the viscosity-measure solutions}
\label{Sec:preliminaries}
As mentioned in the introduction, a solution of \eqref{HJ}--\eqref{CE} is intended in the viscosity sense for \eqref{HJ}, while in the sense of Poupaud-Rascle \cite{PR} for \eqref{CE}. Concerning the definition of viscosity solutions, we refer to the pioneering articles \cite{CL,CIL}. Here, for the reader's convenience, we recall the usual assumptions on $H$ and the consequent general existence and uniqueness result for \eqref{HJ} that we shall use in the sequel.

Let us define $Q_T:=\R^d\times [0,T]$, $T>0$, $d\ge1$, and let the hamiltonian $H$ satisfy the following:\hfill\break
$(\mathbf{H}_1)$ $H$ is uniformly continuous on $Q_T\times B(0,R)$ for any $R>0$ ;\hfill\break
$(\mathbf{H}_2)$ $H(x,t,0)$ is uniformly bounded : $\sup_{Q_T}|H(x,t,0)|\equiv\,M<+\infty$ ;\hfill\break
$(\mathbf{H}_3)$ there exists $\eta>0$ s.t. : $|H(x,t,p)-H(y,t,p)|\le\eta\,(1+|p|)|x-y|\,,\  \forall\ t\in[0,T]\,,\ \forall\ x,y,p\in\R^d$ .

\begin{theorem}[\cite{CL}]
Under hypothesis $(\mathbf{H}_1)$, $(\mathbf{H}_2)$ and $(\mathbf{H}_3)$, if in addition the initial data $u_0$ belongs to $(W^{1,\infty}\cap BUC)(\R^d)$, there exists a unique viscosity solution $u\in(W^{1,\infty}\cap BUC)(Q_T)$ of \eqref{HJ}, \eqref{IC}.
\label{th:HJ}
\end{theorem}

As stated in Theorem \ref{th:HJ}, the expected regularity for $\nabla u$ is $L^\infty (Q_T)$. Therefore, the characteristics $X(t;x)$ associated to the conservative transport equation \eqref{CE} cannot be defined as classical or distributional solutions of the system below
\begin{equation}
\left\{
\begin{array}{l}
\partial_t X (t;x)=a(X(t;x),\nabla u(X(t;x),t))\\[5pt]
X(0;x)=x\,,\quad\quad x\in\R^d\,,
\end{array}
\right.
\label{flow}
\end{equation}
but have to be understood in a generalized sense. Once the flow $X(t;x)$ is uniquely defined and continuous on $[0,T]\times\R^d$, the natural definition of a solution of the conservation law \eqref{CE} with a given initial data $m_0$ belonging to the space of bounded measures $\cM_1(\R^d)$ is, see\cite{PR},
\begin{equation}
m(t)=X(t\,;\cdot)\,{\#}\,m_0\,.
\label{PRsol}
\end{equation}
Equation \eqref{PRsol} means that the measure solution $m(t)$ is the image (or the push-forward) of $m_0$ by  the flow $X(t\,;\cdot)$, i.e. for any Borel set $B\subset\R^d$,
\[
m(t)(B)=m_0(X(t\,;\cdot)^{-1}(B))=m_0(\{x\in\R^d:\, X(t;x)\in B\})\,,
\]
or equivalently
\[
\langle m(t),\phi\rangle= \langle m_0,\phi(X(t;\cdot))\rangle\,,\qquad
\text{for any } \phi\in C^0_0(\R^d)\,.
\]

It turns out that the definition of Filippov characteristics  is well suited and that with it,  \eqref{PRsol} perfectly make sense as soon as the characteristics are unique.
 A definition of Filippov characteristics solution of \eqref{flow} equivalent to the original one and easier to handle is given in \cite{PR},  as follows:
\begin{definition}[Filippov characteristics]
For any given $x\in\R^d$ and $T>0$, a generalized solution in the sense of Filippov of system \eqref{flow} is a continuous map $X(\cdot\,;x)\,:\,[0,T]\mapsto\R^d$ which satisfies
\[
\{\,\underline{\cM}(a(\cdot,\nabla u(\cdot,t))\cdot v)\}(X(t;x))\le\partial_t X(t;x)\cdot v\le\{\,\overline{\cM}(a(\cdot,\nabla u(\cdot,t))\cdot v)\}(X(t;x))\,,\quad a.e.\ t\in[0,T]\,,
\]
for any $v\in\R^d$, where
\begin{align*}
\{\,\underline{\cM}(a(\cdot,\nabla u(\cdot,t))\cdot v)\}(z):=\sup_{r>0}\left(\einf_{y\in B(z,r)}a(y,\nabla u(y,t))\cdot v\right)\\
\{\,\overline{\cM}(a(\cdot,\nabla u(\cdot,t))\cdot v)\}(z):=\inf_{r>0}\left(\esup_{y\in B(z,r)}a(y,\nabla u(y,t))\cdot v\right)\,.
\end{align*}
\label{def:Filippov}
\end{definition}
Then, assuming
\begin{itemize}
\item[$(\mathbf{A}_1)$] $|a|\in L^\infty_x(\R^d\,;\,L^\infty_{loc,\,p}(\R^d))$,
\end{itemize}
the vector field $a(\cdot,\nabla u)$ satisfies \eqref{weakhyp} and the Filippov characteristics of system \eqref{flow} exist. Moreover, if the following \emph{one-sided Lipschitz condition} (OSL) holds true
\begin{equation}
\left(a(x,\nabla u(x,t))-a(y,\nabla u(y,t))\right)\cdot(x-y)\le \gamma(t)|x-y|^2\quad
\hbox{a.e. $t\in[0,T]$, $x,y\in\R^d$}
\label{OSL}
\end{equation}
for a function $\gamma\in L^1([0,T])$, these characteristics are unique, the flow $(t,x)\mapsto X(t;x)$ is continuous on $[0,T]\times\R^d$, the map  from $\R^d$ to $\R^d$, $x\mapsto X(t;x)$ is onto; the following existence and uniqueness result of a viscosity-measure solution of \eqref{HJ}-\eqref{CE}-\eqref{IC} can be stated:
\begin{theorem}
Let $u_0\in(W^{1,\infty}\cap BUC)(\R^d)$ and $m_0\in\cM_1(\R^d)$. Assume $(\mathbf{H}_1)$-$(\mathbf{H}_2)$-$(\mathbf{H}_3)$ and $(\mathbf{A}_1)$. Then,  if the viscosity solution $u$ satisfies \eqref{OSL}, there exists a unique measure solution $m\in C^0([0,T];\cM_1(\R^d)-weak\,*)$ of \eqref{CE}-\eqref{IC} given by \eqref{PRsol}.
\label{th:HJ-CL}
\end{theorem}

In the reference case $a(x,p)=p$, as in \eqref{SH}, requiring the (OSL) condition \eqref{OSL} is equivalent to require that the viscosity solution $u$ be semiconcave with respect to~$x$,~i.e.
\begin{equation}
u(x+y,t)-2\,u(x,t)+u(x-y,t)\le \beta(t)\,|y|^2\,,
\label{semiconcavity}
\end{equation}
 for all $x,y\in\R^d$, $t\in[0,T]$ and for some $\beta\in L^1([0,T])$.
The semiconcavity property \eqref{semiconcavity} often characterizes the viscosity solution of Hamilton-Jacobi equations arising in control problems.
 However, for general vector fields $a$, it does not imply the necessary condition \eqref{OSL}. Indeed, for $u_\e=u*\rho_\e$, where $\rho_\e$ is the mollifier introduced above, the semiconcavity \eqref{semiconcavity} implies:
$\nu^T\,D^2u^\e(x,t)\,\nu\le\beta(t)$, for all $\nu\in\R^d$ s.t. $|\nu|=1$ and for all $(x,t)\in Q_T$, (see \cite{CS}).
  Moreover, $\nabla u^\e(x,t)\to\,\nabla u(x,t)$ as $\e\to0$, for a.a. $(x,t)\in Q_T$. Owing to this convergence property,
it is sufficient to obtain \eqref{OSL} for $u^\e$, whenever $a$ is at least continuous w.r.t. $p$.
 Assuming in addition that $a$ is differentiable w.r.t. $p$ and satisfies a one sided Lipschitz condition w.r.t. $x$ locally uniformly in $p$, we have
$$
a(x,p)-a(y,q)=\int_0^1D_pa(x,q+s(p-q))(p-q)\,ds+a(x,q)-a(y,q)\,,
$$
and the estimate
\begin{equation}
\begin{array}{l}
\left(a(x,\nabla u^\e(x,t))-a(y,\nabla u^\e(y,t))\right)\cdot(x-y)\le\\
\quad\quad\quad\int_0^1ds\int_0^1d\sigma\,(x-y)^T\,\Gamma(s)\,D^2_xu^\e(y+\sigma(x-y),t)(x-y)
+ C\,|x-y|^2\,,
\end{array}
\label{est:OSL}
\end{equation}
with $\Gamma(s)=D_pa(x,\nabla u^\e(y,t)+s(\nabla u^\e(x,t)-\nabla u^\e(y,t)))$. Therefore, the (OSL) condition follows easily from \eqref{semiconcavity}, if either $D_pa(x,p)=I$, as expected, or $d=1$ and
$\partial_pa(x,p)$ is non negative and upper bounded, as for \eqref{BS}. Apart from these two cases (already considered in \cite{BK}), the (OSL) condition is an assumption and has to be verified for the specific model at hand.

We conclude this section by listing the properties that the measure solution \eqref{PRsol} satisfies, see \cite{PR} :\hfill\break
(i) {\sl monotonicity} : $m_0\ge \nu_0\,\Rightarrow X(t\,;\cdot)\,{\#}\,m_0\ge X(t\,;\cdot)\,{\#}\,\nu_0$;\hfill\break
(ii) {\sl mass conservation} : $m(t)(\R^d)=m_0(\R^d)$;\hfill\break
(iii) {\sl contraction property} : $|m(t)|(B)\le |m_0|(B)$ , for any Borel set $B\subseteq\R^d$;\hfill\break
(iv) {\sl semi-group property} : $X(t-s\,;\cdot)\,{\#}\left(X(s\,;\cdot)\,{\#}\,m_0\right)=X(t\,;\cdot)\,{\#}\,m_0$;\hfill\break
(v) {\sl uniform compactness at infinity} : $\forall\ \e>0$ there exists $R>0$ s.t. $|m(t)|(\R^d\setminus B_R(0))\le\e$,  $\forall\ t\in[0,T]$.\hfill\break
It is worth to underline that the {\sl uniform compactness at infinity} property (v) is fundamental  to prove the convergence of approximate measure solutions toward the exact one in $C^0([0,T];\cM_1(\R^d)\,w-*)$.
%
\section{The semi-lagrangian scheme}
\label{Sec:SLscheme}
This section is devoted to the construction and the convergence analysis of a semi-lagrangian scheme for the time discretization of system \eqref{HJ}-\eqref{CE} over the time interval $[0,T]$. For an introduction to this class of schemes we refer to   \cite[Appendix B]{bcd} and \cite{CFN,FF,FF2,Gs}. To proceed, we need to refine the previous hypothesis on the hamiltonian $H$ and to add assumptions on the growth of $H$ w.r.t. $p$. Therefore, let us assume in this section $(\mathbf{H}_1)$, $(\mathbf{H}_3)$~and
\begin{itemize}
\item[$(\mathbf{H}_2')$] for any $R>0$,  $\sup_{Q_T\times\overline B(0,R)}|H(x,t,p)|\equiv\,M(R)<+\infty$;
\item [$(\mathbf{H}_4)$] $H$ is convex in $p$ and either linear at infinity (i) or  superlinear (ii):
\begin{itemize}
\item [(i)] $\exists\ K>0$ s.t. $\lim\limits_{|p|\to\infty}\f{H(x,t,p)}{|p|}= K$, uniformly in $(x,t)\in Q_T$, and there exists a positive function $\alpha\in L^\infty(\R^d)$ such that $|H(x,t,p)-H(y,t,p)|\le\alpha(p)|x-y|$ uniformly on $Q_T\times\R^d$\,;
\item [(ii)] $\lim\limits_{|p|\to\infty}\,\inf\limits_{Q_T}\,\f{H(x,t,p)}{|p|}=+\infty$ .
\end{itemize}
\end{itemize}
\begin{remark} It is worth noticing that the hamiltonians \eqref{SH} and \eqref{BS} satisfy all the required assumptions as soon as the potential $V$ is uniformly continuous and bounded over $Q_T$ and Lipschitz continuous in $x$ uniformly in $t$. The growth assumption $(\mathbf{H}_4)$-(i) can be relaxed to include also hamiltonians that are not uniformly positive at infinity. But we leave this generalization to the reader.
\end{remark}
%
\subsection{The semi-lagrangian scheme for the (HJ) equation}
\label{Sec:HJsemilagrangian}
Let $N\in\N$ be fixed and let $h=T\,N^{-1}$ be the associated time step used in the semi-discrete scheme to be defined. Since the hamiltonian $H$ is convex in $p$ and continuous (in fact lower semi-continuity would be sufficient \cite{Ro}), then $H(x,t,p)=H^{**}(x,t,p)$, where $H^*$ is the Legendre transform of $H$ with respect to $p$, i.e. $H^*(x,t,\xi)=\sup_{p\in\R^d}\{\xi\cdot p-H(x,t,p)\}$, and the (HJ) equation can be written as
\begin{equation}
\partial_tu(x,t)=\inf_{\xi\in\R^d}\{-\xi\cdot\nabla u(x,t)+H^*(x,t,\xi)\}.
\label{HJ**}
\end{equation}
Next, plugging the first order forward finite difference for the approximation of $\partial_tu$
$$
\partial_tu(x,t)\sim\f{u(x,t+h)-u(x,t)}h
$$
 and the first order approximation of the directional derivative $-\xi\cdot\nabla u$
\begin{equation}
-\xi\cdot\nabla u(x,t)\sim\f{u(x-\xi h,t)-u(x,t)}h\,,
\label{eq:OSDD}
\end{equation}
 into \eqref{HJ**}, one easily obtains the following first order approximation for~$u(x,t+h)$
$$
u(x,t+h)\sim\inf_{\xi\in\R^d}\left\{u(x-\xi\,h,t)+h\,H^*(x,t,\xi)\right\}\,.
$$
Let us observe that whenever the exact solution $u$ of (HJ) is semiconcave in $x$, then it possesses one-sided directional derivatives at any $x$ and in any direction, i.e. the limit as $h\searrow0$ of the right hand side of \eqref{eq:OSDD} always exists. Moreover, if $u$ is differentiable w.r.t. $x$, the previous limit coincides with the left hand side of~\eqref{eq:OSDD}. In other word, the approximation \eqref{eq:OSDD} is consistent.

For the semi-discrete scheme, it is sufficient to  inductively define the approximation $u^n=u^n(x)$ of the exact solution $u$ of (HJ) at time $t^n=n\,h$ for $n=0,\dots,N$, by
\begin{equation}
u^{n+1}(x)=\inf_{\xi\in\R^d}\left\{u^n(x-\xi h)+h\,H^*(x,t^n,\xi)\right\}\,,\quad x\in\R^d\,,
\label{HJn}
\end{equation}
with $u^0=u_0$, the initial data for (HJ) in \eqref{IC}. Next, we need to show that $u^n$ as given by \eqref{HJn} shares the  properties of $u_0$, in the same way as the exact viscosity solution $u$ does. These are well established facts in the context of the approximation of functions by the inf-convolution operator (see for example \cite{CS}). However, the difficulties here are on the one hand to show that the properties of the initial data $u_0$ are propagated to $u^n$ uniformly with respect to $n$ and $h$, and on the other hand to handle the dependency of $H$ on~$(x,t)$.

Let us define $Q_h:=\R^d\times\{0,\dots,N\}$ and the set of the arguments associated to the infimum in  \eqref{HJn}

\begin{equation}
A^n(x):=\arg\inf_{\xi\in\R^d}\left\{u^n(x-\xi h)+h\,H^*(x,t^n,\xi)\right\}\,,\quad (x,n)\in Q_h\,.
\label{arginf}
\end{equation}
%


%
\begin{lemma}[Properties of $H^*$]
Let $H$ satisfy $(\mathbf{H}_1)$, $(\mathbf{H}_2')$ and $(\mathbf{H}_3)$. If in addition $H$ satisfies $(\mathbf{H}_4)$-(i), then :
\begin{itemize}
\item [(a)] $H^*(x,t,\xi)=+\infty$ if $|\xi|>K$, for any $(x,t)\in Q_T$ , $H^*(\cdot,\cdot,0)\in L^\infty(Q_T)$ and $H^*$ is Lipschitz continuous in $x$ uniformly in $(\xi,t)\in \overline B(0,K)\times[0,T]$.
\end{itemize}
On the other hand, if in addition $H$ satisfies $(\mathbf{H}_4)$-(ii), then :
\begin{itemize}
\item [(b)] $H^*$ also satisfies $(\mathbf{H}_4)$-(ii), for any $r>0$ there exists $R=R(r)>0$ such that
\begin{equation}
H^*(x,t,\xi)=\max\limits_{p\in\overline B(0,R)}\{p\cdot\xi-H(x,t,p)\}\,,\quad\quad
\forall\ (x,t,\xi)\in Q_T\times\overline B(0,r)\,,
\label{eq:H^*}
\end{equation}
$H^*\in L^\infty(Q_T\times \overline B(0,r))$ and $H^*$ is Lipschitz continuous in $x$ uniformly in $(\xi,t)\in\overline B(0,r)\times[0,T]$.
\end{itemize}
\label{lm: H^*properties}
\end{lemma}
The proof of the Lemma above is given in the Appendix. Let us just point out here that we have chosen to take the hamiltonian $H$ to be Lipschitz continuous in $x$ uniformly in $p$ and $t$ when $H$ is linear in $p$ for $|p|\to+\infty$, in order to obtain the Lipschitz continuity of $H^*$ w.r.t. $x$. Indeed, in this case the supremum defining $H^*$ is not a priori reached. Therefore we cannot conclude as in case $(\mathbf{H}_4)$-(ii).

\begin{lemma}[Properties of $u^n$]
Let $u_0\in W^{1,\infty}(\R^d)$. Then, under hypothesis $(\mathbf{H}_1)$, $(\mathbf{H}_2')$, $(\mathbf{H}_3)$ and $(\mathbf{H}_4)$, $u^n$ is well defined over $Q_h$  by \eqref{HJn}, i.e. $A^n(x)$ is a non empty bounded set of $\R^d$ uniformly in $(x,n)\in Q_h$ and the infimum is a minimum. Moreover,  $u^n\in W^{1,\infty}(\R^d)$ and $\|u^n\|_{ W^{1,\infty}(\R^d)}$ is bounded uniformly in $n$ and, assuming $H^*\in L^\infty(Q_T\times\overline B(0,K))$ when the growth condition $(\mathbf{H}_4)$-$(i)$ is satisfied, $u^n$ is Lipschitz continuous with respect to $n$ uniformly in $x$, i.e. there exist three positive constants $C_0$, $C_1$ and $C_2$ independent of $h$ and $n$, s.t.
\begin{equation}
\|u^n\|_{L^\infty(\R^d)}\le C_0\,,\quad \|\nabla u^n\|_{L^\infty(\R^d)}\le C_1\quad\text{and}\quad
|u^n(x)-u^m(x)|\le C_2|n-m|\,h\,,
\label{propertiesun}
\end{equation}
for all $n,m=0,\cdots,N$, and $x\in\R^d$.
\label{lm: u^nproperties}
\end{lemma}
\begin{proof}
We suppose that $u^n\in W^{1,\infty}(\R^d)$ with
$$
\|u^n\|_{L^\infty(\R^d)}\le C_0^n\quad\text{and}\quad \|\nabla u^n\|_{L^\infty(\R^d)}\le C_1^n\,,
$$
where $C_0^n$ and $C_1^n$ are independent of $h$. Then, the proof will follow by induction on $n$.

It follows immediately by Lemma \ref{lm: H^*properties} that $u^{n+1}$ is upper bounded since
$$
u^{n+1}(x)\le u^n(x)+h\,H^*(x,t^n,0)\le C_0^n+h\,\|H^*(\cdot,\cdot,0)\|_{L^\infty(Q_T)},
\quad\forall x\in\R^d\,.
$$
Moreover, $u^{n+1}$ is obviously lower bounded since $u^n$ is bounded and $H^*(x,t,\xi)\ge -M$, for any $(x,t,\xi)\in Q_T\times\R^d$, where $M \equiv\sup_{Q_T}|H(x,t,0)|$, so that
$$\|u^{n+1}\|_{L^\infty(\R^d)}\le C_0^n+h\max\{M,\,\|H^*(\cdot,\cdot,0)\|_{L^\infty(Q_T)}\}\,.$$

Next, if $H$ satisfies $(\mathbf{H}_4)$-(i), $A^n(x)\subset\overline B(0,K)$ for any $x\in\R^d$ and the infimum in \eqref{HJn} is attained due to the continuity of $u^n(x-\xi h)+h\,H^*(x,t^n,\xi)$ w.r.t. $\xi$.

On the other hand, if $H$ satisfies $(\mathbf{H}_4)$-(ii), since $u^n$  and $H^*(x,t,0)$ are bounded and $H^*$ is superlinear, there exists $R^n>0$ (increasing w.r.t. $n$ but upper bounded uniformly in $n$ and $h$) s.t. $A^n(x)\subset\overline B(0, R^n)$ for any $x\in\R^d$ and again the infimum in \eqref{HJn} is attained.

Let us prove now that $u^{n+1}$ is Lipschitz continuous. We have for any $x,y\in\R^d$ and for $\alpha^n(x)\in A^n(x)$
$$
u^{n+1}(x)=u^n(x-h\,\alpha^n(x))+h\,H^*(x,t^n,\alpha^n(x))
$$
and
$$
u^{n+1}(y)\le u^n(y-h\,\alpha^n(x))+h\,H^*(y,t^n,\alpha^n(x))\,.
$$
Therefore,
$$
u^{n+1}(y)-u^{n+1}(x)\le\left[\|\nabla u^n\|_\infty+h\,L_{H^*}\right]|x-y|
\le\left[C_1^n+h\,L_{H^*}\right]|x-y|\,,
$$
where $L_{H^*}$ is the Lipschitz constant of $H^*$ w.r.t. $x$, and the statement follows exchanging the role of $x$ and $y$. Consequently, $u^{n+1}\in W^{1,\infty}(\R^d)$.

It remains to prove that $u^{n}$ is Lipschitz continuous with respect to $n$. As before, we have for any $x\in\R^d$ and the corresponding argument $\alpha^n(x)\in A^n(x)$
\begin{align}
\nonumber
|u^{n+1}(x)-u^n(x)|&=|u^n(x-h\,\alpha^n(x))-u^n(x+h\,H^*(x,t^n,\alpha^n(x))|
\\ \nonumber
&\le\|\nabla u^n\|_\infty\,h\,|\alpha^n(x)|+h|H^*(x,t^n,\alpha^n(x))|\le C_1^nh\,|A^n|+h|H^*(x,t^n,\alpha^n(x))|\,,
\end{align}
and the proof is completed, thanks to the uniform boundedness of $A^n$ and $H^*$.
\end{proof}
%


Let us observe that when the hamiltonian $H$ grows linearly at infinity, i.e. when the growth condition $(\mathbf{H}_4)$-$(i)$ is satisfied, then $H^*$ is not necessarily upper bounded. Therefore, the additional hypothesis $H^*\in L^\infty(Q_T\times\overline B(0,K))$ in Lemma \ref{lm: u^nproperties} is necessary. It is easily seen that this additional hypothesis is satisfied by the hamiltonian \eqref{BS}.

Finally, in order to obtain semiconcavity, we need the following semiconcavity hypothesis on $H^*$.

\begin{itemize}
\item[$(\mathbf{H}_5)$] $H^*=H^*(x,t,\xi)$ is semiconcave in $x$ uniformly in $t$ and $\xi$, with semiconcavity constant
$C_{conc}^{H^*}$.
\end{itemize}
\begin{lemma}[Semiconcavity]
Let $u_0\!\in\! W^{1,\infty}(\R^d)$ be semiconcave. Then, under hypothesis $(\mathbf{H}_1)$, $(\mathbf{H}_2')$, $(\mathbf{H}_3)$, $(\mathbf{H}_4)$ and $(\mathbf{H}_5)$, $u^n$ is also semiconcave uniformly in $h$ and~$n$.
\label{lm:semiconcavity}
\end{lemma}
\begin{proof}
We proceed as in Lemma \ref{lm: u^nproperties} supposing that $u^n$ is semiconcave with semiconcavity constant $C_{conc}^n$ . Then, for any $x\in\R^d$, any corresponding argument $\alpha^n(x)\in A^n(x)$ and any $y\in\R^d$, we have that
\begin{align}
&u^{n+1}(x+y)-2u^{n+1}(x)+u^{n+1}(x-y)
\nonumber\\
&\qquad\le u^{n}(x+y-h\,\alpha^n(x))-2\,u^{n}(x-h\,\alpha^n(x))+u^{n}(x-y-h\,\alpha^n(x))
\nonumber\\
&\qquad\qquad+h\left[H^*(x+y,t^n,\alpha^n(x))-2\,H^*(x,t^n,\alpha^n(x))+H^*(x-y,t^n,\alpha^n(x))\right]
\nonumber\\
&\qquad\le\left[C_{conc}^n+h\,C_{conc}^{H^*}\right]|y|^2\,.
\nonumber
\end{align}
Hence, $u^{n+1}$ is semiconcave with $C_{conc}^{n+1}\le C_{conc}^n+h\,C_{conc}^{H^*}$, and by induction $u^n$ is semiconcave for all $n=1,\dots,N$ with : $C_{conc}^{n}\le C_{conc}^{u_0}+T\,C_{conc}^{H^*}$\,.

\end{proof}
\begin{remark}
Assumption $(\mathbf{H}_5)$ is necessary since the hamiltonian $H$ depends on $x$ and $t$. It can be satisfied under regularity hypothesis on $H$. In the case of the hamiltonians \eqref{SH} and \eqref{BS}, $(\mathbf{H}_5)$ is satisfied if the potential $V$ is convex in $x$ uniformly in $t$.
\end{remark}
%
\subsection{The semi-lagrangian scheme for the (CE) equation}
\label{Sec:CLsemilagrangian}
Now we can proceed to the construction of a semi-discrete scheme for the transport equation \eqref{CE}, replacing the time continuous flow $X(t;x)$ with a time discrete one.
Set $u^n_\e=u^n*\rho_\e$ and let us consider the following implicit backward Euler scheme
\begin{equation}
\left\{
\begin{array}{l}
X^{n+1}=X^n+h\,a(X^{n+1},\nabla u^{n+1}_\e(X^{n+1}))\,,\quad \quad n=0,\cdots,N-1,\\
X^0=x\,,\quad\quad x\in\R^d\,.
\end{array}
\right.
\label{eq:Eulerimplicit}
\end{equation}

Next, given the initial measure $m_0$ and replacing \eqref{flow} with \eqref{eq:Eulerimplicit}, we define the semi-discrete approximation of the measure solution \eqref{PRsol} as
\begin{equation}
m^n=X^n(\cdot)\,{\#}\,m_0\,,\quad\quad \text{for } n=0,\dots,N\,,
\label{mn}
\end{equation}
i.e. for any Borel set $B\subset\R^d$
\[
m^n(B)=m_0((X^n)^{-1}(B))=m_0(\{x\in\R^d:\, X^n(x)\in B\})\,,
\]
or equivalently
$$
\langle m^n,\phi\rangle= \langle m_0,\phi(X^n(\cdot))\rangle\qquad \text{for any $\phi\in C^0_0(\R^d)$}.
$$

Above, the dependence of $X^n$ on $x$ and $\e$ and consequently of $m^n$ on $\e$  has been skipped to simplify the notations, but it will be made explicit when necessary below.

It is worth noticing that it is not possible to define the semi-discrete flow \eqref{eq:Eulerimplicit} with $\nabla u^{n+1}$ instead of the gradient of the regularized $u^{n+1}_\e$, since we need $X^n$ to be defined for all $(x,n)\in Q_h$. Moreover, it is also not possible to use an explicit forward Euler scheme to define the trajectories $X^n$, because the discrete (OSL) condition \eqref{OSLh} does not provide the necessary
equicontinuity of these trajectories, in contrast with the implicit one \eqref{eq:Eulerimplicit}, (see Lemma \ref{lm:Xh} and Remark \ref{rm:explicit scheme} below).

The following Lemma gives us the existence, uniqueness and the regularity properties of $X^n$ necessary to the  well posedness of $m^n$ and to pass to the limit as $h$ and $\e$ go to 0. Owing to the implicit nature of the Euler scheme \eqref{eq:Eulerimplicit}, we have to assume an upper bound on the space-time step $h$.
\begin{lemma}[Properties of $X^n$]
Let $u_0\in W^{1,\infty}(\R^d)$ and assume $(\mathbf{H}_1)$, $(\mathbf{H}_2')$, $(\mathbf{H}_3)$, $(\mathbf{H}_4)$ and~$(\mathbf{A}_1)$. Then, if in addition $a\in (C^0(\R^d\times\R^d))^d$,  the ($\hbox{OSL}_h$) condition
\begin{equation}
\left(a(x,\nabla u^n_\e(x))-a(y,\nabla u^n_\e(y))\right)\cdot(x-y)\le C|x-y|^2,\quad x,y\in\R^d\,,n=0,\cdots,N\,,
\label{OSLh}
\end{equation}
is satisfied, with constant $C$ independent of $\e$ and $h$, and $C\,h<1$, the solution of \eqref{eq:Eulerimplicit} is univocally defined. Moreover, the flow $(n,x)\mapsto X^n(x)$ is locally bounded and Lipschitz continuous w.r.t.~$(x,n)\in~Q_h$, uniformly in $\e$.
\label{lm:Xh}
\end{lemma}
\begin{proof} The existence of $X^{n+1}$ given $X^n$, is insured by the Brouwer fixed point theorem, the map $y~\mapsto~X^n+h\,a(y,\nabla u^{n+1}_\e(y))$ being continuous and bounded. The uniqueness of $X^{n+1}$ and the Lipschitz continuity w.r.t.~$x$ follow both from \eqref{OSLh} and the upper bound on $h$. Indeed, for $X^{n+1}$ and $Y^{n+1}$ defined by \eqref{eq:Eulerimplicit} we~have
\begin{align}
&|X^{n+1}-Y^{n+1}|^2=(X^{n+1}-Y^{n+1})\cdot(X^n-Y^n)
\nonumber\\
&\quad\quad+h\,(X^{n+1}-Y^{n+1})\cdot\left(a(X^{n+1},\nabla u^{n+1}_\e(X^{n+1}))-a(Y^{n+1},\nabla u^{n+1}_\e(Y^{n+1}))\right)
\nonumber\\
&\quad\quad\le|X^{n+1}-Y^{n+1}||X^n-Y^n|+C\,h\,|X^{n+1}-Y^{n+1}|^2\,,
\label{eq:estequicontinuity}
\end{align}
i.e.
$$
|X^{n+1}-Y^{n+1}|\le\left(1+\f{C\,h}{1-C\,h}\right)|X^n-Y^n|\,.
$$
Taking $X^n=Y^n$ we get the uniqueness, while iterating over $n$ we get for two starting points $x,\ y\in\R^d$
$$
|X^{n}(x)-X^{n}(y)|\le\left(1+\f{C\,h}{1-C\,h}\right)^{n}|x-y|\,.
$$
Therefore, for $\delta\in(0,1)$ s.t. $C\, h\le 1-\delta$, we have
$$
|X^{n}(x)-X^{n}(y)|\le\exp(n\,C\,h\,\delta^{-1})|x-y|\le\exp(CT\,\delta^{-1})|x-y|\,,
\quad n=0,\cdots,N\,.
$$

Finally, for any $x\in\R^d$ it holds true that
$$
|X^{n}(x)-X^m(x)|\le h|n-m|\sup\limits_{x\in\R^d}\,\sup\limits_{p\in \overline B(0,C_1)}|a(x,p)|\,,
$$
where $C_1$ is given in \eqref{propertiesun}, and we obtain the Lipschitz continuity w.r.t. $n$. The above estimate gives us also that the image by $X^n$ of $B(0,R)$ is contained in the ball of radius $(R+T\sup\limits_{x\in\R^d}\,\sup\limits_{p\in \overline B(0,C_1)}|a(x,p)|)$. Therefore, $X^n$ is locally bounded uniformly in $\e$ and $n$ and the Lemma is proved.
\end{proof}
\begin{lemma}[Properties of $m^n$] Let $m_0\in\cM_1(\R^d)$. Under the same hypothesis as in Lemma \ref{lm:Xh}, the approximated solution $m^n$ in \eqref{mn} is well defined in $\cM_1(\R^d)$. Moreover, $m^n$ satisfies the same properties (i)-(v) as the exact measure solution $m$, uniformly in $n$.
\label{lm:mhe}
\end{lemma}
\begin{proof} By Lemma \ref{lm:Xh}, the map $x\mapsto X^n(x)$ is uniquely defined and continuous over $\R^d$. Moreover, it is onto from $\R^d$ to $\R^d$. Therefore, the approximated solution $m^n$ is well defined in $\cM_1(\R^d)$. Furthermore, it is easy to see that $m^n$ satisfies all the properties (i)-(v) of the exact measure solution~$m$. Concerning~(v), it follows from the local and uniform in $n$ boundedness of $X^n$ proved above (see Lemma 3.1 in~\cite{PR}).
\end{proof}
%
\subsection{The convergence}
\label{Sec:convergence}
We can finally discuss the convergence of the semi-discrete scheme \eqref{HJn}-\eqref{eq:Eulerimplicit}-\eqref{mn} to the solution of the continuous problem \eqref{HJ}-\eqref{CE}-\eqref{IC}. The key tools to prove the convergence result have been given in~\cite{PR}, where the authors analyzed the stability of the measure solution with respect to perturbations of the vector field $a$. The stability of the Filippov characteristics is of course the basic stone. Later, these tools have been  adapted  in \cite{BK} to analyze the stability of the viscosity-measure solution of \eqref{HJ}-\eqref{CE}-\eqref{IC} with $a(x,p)=p$, with respect to the vanishing viscosity perturbation of the (HJ) equation.

To discuss the convergence, we define the  piecewise constant w.r.t. time approximated solution
$$
u_h(x,t)=u^{[t/h]}(x)\,,\quad (x,t)\in Q_T\,, \quad\quad
m_h^\e(t)=m^{[t/h]}\,,\quad t\in[0,T]\,,
$$
and $u_h^\e=u_h*\rho_\e$. Moreover, we also need to define  the following time continuous trajectories by linear interpolation of $X^n(x)$ and $X^{n+1}(x)$, for any $x\in\R^d$,
\begin{equation}
X_h^\e(t;x)=X^n(x)+(t-t^n)\,a(X^{n+1}(x),\nabla u^{n+1}_\e(X^{n+1}(x)))\,,\quad t\in[t^n,t^{n+1}]\,,\quad n=0,\cdots,N\,.
\label{eq:Xh}
\end{equation}
Note that \eqref{eq:Xh} gives us time continuous trajectories $X_h^\e$ with the same regularity as $X^n$, i.e. $X_h^\e$ is locally bounded and Lipschitz continuous w.r.t. $(t,x)\in[0,T]\times\R^d$, uniformly in $\e$ and $h$.


%
\begin{theorem}[Convergence]
Let $u_0\in(W^{1,\infty}\cap BUC)(\R^d)$ be semiconcave, $m_0\in\cM_1(\R^d)$ and assume $(\mathbf{H}_1)$, $(\mathbf{H}_2')$, $(\mathbf{H}_3)$, $(\mathbf{H}_4)$, $(\mathbf{H}_5)$ and $(\mathbf{A}_1)$. Assume in addition that $a\in (C^0(\R^d\times\R^d))^d$, $u^n_\e$~satisfies the ($\hbox{OSL}_h$) condition \eqref{OSLh} and $Ch<1$. Let $\e=\e(h)$ be s.t. $\e(h)\to0$ as $h\to0$. Then, as $h\to0$,
\begin{itemize}
\item[(i)] $u_h\to u$ in $L^\infty(Q_T)$ and $\nabla u_h^{\e}(x,t)\to \nabla u(x,t)$ at any point $(x,t)\in Q_T$ of differentiability of $u$, where $u$ is the unique viscosity solution of~\eqref{HJ};
\item[(ii)] $X_h^\e$ converges locally uniformly in $Q_T$ toward the unique Filippov characteristic $X$ associated to the vector field $a(\cdot,\nabla u)$;
\item[(iii)] $m_h^{\e} \rightharpoonup m$ in $C^0([0,T];\cM_1(\R^d)\,w-*)$, where $m$ is the unique measure solution of \eqref{CE}.
\end{itemize}
\label{th:convergence}
\end{theorem}
\begin{proof} Following standard results of the viscosity solution theory (see for instance \cite{So}), it can be proved that there exists a constant $C$ independent of $h$,  s.t.
\begin{equation}
\|u^n-u(t^n)\|_{L^\infty(\R^d)}\le C\,h^{1/2}\,,\quad\quad n=0,\cdots, N\,.
\label{eq:estun}
\end{equation}
Next, let $(x,t)\in Q_T$ be a point of differentiability of $u$. For $n=\left[\f th\right]$, by Taylor expansion and the semiconcavity of~$u^{n}_\e$, we have that
$$
u^n_\e(y)-u^n_\e(x)-\nabla u^n_\e(x)\cdot (y-x)\le(C^{u_0}_{conc}+T\,C^{H^*}_{conc}) |y-x|^2\,,\quad\forall\ y\in\R^d\,.
$$
Therefore, since $\nabla u^n_\e(x)\to\,p$ as $h\to0$, by the uniform bound of $\nabla u^n_\e$, the previous inequality and the convergence of $u^n_\e(t)$ toward $u(t)$, gives us
$$
u(y,t)-u(x,t)-p\cdot (y-x)\le (C^{u_0}_{conc}+T\,C^{H^*}_{conc})(t)|y-x|^2\,,
$$
i.e. $p$ belongs to the subdifferential of $u$ at $x$. Being $u$ differentiable at $(x,t)$, $p=\nabla u(x,t)$ and we have obtained the proof of (i).

We now prove the second part of the statement. From Lemma \ref{lm:Xh} it follows immediately that, up to a subsequence, $X_h^\e$ converges uniformly on every compact set of $Q_T$ to a continuous function $Y$, as $h\to0$. Next, since by (i) and the ($\hbox{OSL}_h$) condition \eqref{OSLh}, the (OSL) condition \eqref{OSL} is also satisfied, the Filippov characteristic $X$ associated to the vector field $a(\cdot,\nabla u)$ is unique for any $x\in\R^d$. In order to prove that $Y=X$, let us define $\Delta_h^\e(t;x):=X_h^\e(t;x)-X(t;x)$. We are going to prove that $\Delta_h^\e(t;x)\to0$ as $h\to0$. Indeed, we have
\begin{equation}
\f12\partial_t|\Delta_h^\e(t;x)|^2=\partial_tX_h^\e(t;x)\cdot\Delta_h^\e(t;x)-\partial_tX(t;x)\cdot\Delta_h^\e(t;x)\,.
\label{eq:est0}
\end{equation}
By the Definition \ref{def:Filippov} of Filippov characteristics it holds true, for all $x\in\R^d$, a.e. $t\in[0,T]$ and all $r>0$, that
$$
\partial_tX(t;x)\cdot\Delta_h^\e(t;x)\ge \einf_{y\in B(X(t;x),r)}a(y,\nabla u(y,t))\cdot\Delta_h^\e(t;x)\,.
$$
Hence, for all $\delta>0$, there exists $\overline x= \overline x(\delta)\in B(X(t;x),\delta)$ point of differentiability of $u$, s.t.
\begin{equation}
\partial_tX(t;x)\cdot\Delta_h^\e(t;x)\ge a(\overline x,\nabla u(\overline x,t))\cdot\Delta_h^\e(t;x)-\delta\,.
\label{eq:est1}
\end{equation}
Plugging \eqref{eq:est1} into \eqref{eq:est0} and using the definition \eqref{eq:Xh} of $X^{\e}_h$, we obtain for $n=\left[\f th\right]-1$,
\begin{equation}
\begin{array}{lrc}
\f12\partial_t|\Delta_h^\e(t;x)|^2&\le& a(X^{n+1}(x),\nabla u^{n+1}_\e(X^{n+1}(x)))\cdot \Delta_h^\e(t;x)
-a(\overline x,\nabla u(\overline x,t))\cdot\Delta_h^\e(t;x)+\delta\\
&=&\left(a(X^{n+1}(x),\nabla u^{n+1}_\e(X^{n+1}(x)))-a(\overline x,\nabla u^{n+1}_\e(\overline x))\right)\cdot \Delta_h^\e(t;x)\quad\quad\quad\quad\\
&&\quad+\left(a(\overline x,\nabla u^{n+1}_\e(\overline x))-a(\overline x,\nabla u(\overline x,t))\right)\cdot\Delta_h^\e(t;x)+\delta:=I_1+I_2+\delta\,.
\end{array}
\label{eq:est3}
\end{equation}
In order to estimate $I_1$, we decompose $\Delta_h^\e(t;x)$ as
$$
\Delta_h^\e(t;x)=(X_h^\e(t;x)-X^{n+1}(x))+(X^{n+1}(x)-\overline x)+(\overline x-X(t;x))\,,
$$
we note ${\cal A}:=\sup\limits_{x\in\R^d}\,\sup\limits_{p\in \overline B(0,C_1)}|a(x,p)|$, where $C_1$ is given in \eqref{propertiesun}, and we make use of the ($\hbox{OSL}_h$) condition \eqref{OSLh} to get
\begin{equation}
I_1\le 2{\cal A}\,|X_h^\e(t;x)-X^{n+1}(x)|+C\,|X^{n+1}(x)-\overline x|^2+2{\cal A}\,\delta
\le 2{\cal A}^2\,h+C\,|X^{n+1}(x)-\overline x|^2+2{\cal A}\delta\,.
\label{est:I1}
\end{equation}
>From \eqref{est:I1}, \eqref{eq:est3} and the uniform boundedness of $\Delta_h^\e(t;x)$ on every compact subset of $Q_T$, we have obtained
$$
\f12\partial_t|\Delta_h^\e(t;x)|^2\le2{\cal A}^2\,h+C\,|X^{n+1}(x)-\overline x|^2+2{\cal A}\delta
+C|a(\overline x,\nabla u^{n+1}_\e(\overline x))-a(\overline x,\nabla u(\overline x,t))|+\delta\,.
$$
Using the convergence results obtained above and passing to the limit into the previous inequality as $h\to0$ and $\delta\to0$, we get that
$$
\f12\partial_t|Y(t;x)-X(t;x)|\le C\,|Y(t;x)-X(t;x)|^2\,,
$$
i.e. $Y(t;\cdot)=X(t;\cdot)$ in $L^2_{loc}(\R^d)$ a.e. $t\in[0,T]$. Finally, from the continuity of $Y$ and $X$ we deduce that $Y=X$. By the uniqueness of $X$ and the uniform boundedness of $X^\e_h$, we have that all the sequence $X^\e_h$ converge.

It remains to prove the convergence of $m^\e_h$. This can be obtained exactly as in Proposition 3.1 in~\cite{PR} and we skip the proof. We just underline that the strong convergence of the trajectories $X^\e_h$, their time equicontinuity and the uniform compactness at infinity of $m^\e_h$ (property (v)) are the key tools to prove that the sequence $\langle m^\e_h(t),\phi\rangle$ is convergent and equicontinuous on $[0,T]$, for any $\phi\in C^0_0(\R^d)$.
\end{proof}

Proceeding as in \eqref{est:OSL}, one can obtain the ($\hbox{OSL}_h$) condition \eqref{OSLh} for special $a$. This is summarized in the following Corollary which is a straightforward consequence of Theorem~\ref{th:convergence}. It is worth noticing that the Hamiltonians \eqref{SH} and \eqref{BS} enter in the framework of this Corollary.

\begin{corollary} Assume the same hypothesis of Theorem \ref{th:convergence}, except the ($\hbox{OSL}_h$) condition \eqref{OSLh}. If in addition, $a$ is differentiable w.r.t. $p$, satisfies a one sided Lipschitz condition w.r.t. $x$ locally uniformly in $p$ and either $D_pa(x,p)=I$ or $d=1$ and $\partial_pa(x,p)$ is nonnegative and upper bounded, then the same conclusions as in Theorem~\ref{th:convergence} hold true.
\end{corollary}
\begin{remark}[The explicit Euler scheme] The explicit Euler scheme
\begin{equation}
X^{n+1}=X^n+h\,a(X^n,\nabla u^n_\e(X^n))\,,
\label{eq:Eulerexplicit}
\end{equation}
does not provide a family of equicontinuous characteristics, under the ($\hbox{OSL}_h$) condition~\eqref{OSLh}. This is a difficulty naturally intrinsic to the (OSL) condition that allows discontinuity of compressive type only, while the map $x\mapsto X^n(x)$ given by \eqref{eq:Eulerexplicit} can be expansive. Indeed, we have
\begin{align}
&|X^{n+1}-Y^{n+1}|^2=(X^{n+1}-Y^{n+1})\cdot(X^n-Y^n)
\nonumber\\
&\quad\quad+h\,(X^{n+1}-Y^{n+1})\cdot\left(a(X^{n},\nabla u^{n}_\e(X^{n}))-a(Y^{n},\nabla u^{n}_\e(Y^{n}))\right)
\nonumber\\
&\quad\quad\le\f12|X^{n+1}-Y^{n+1}|^2+\f12|X^n-Y^n|^2+C\,h\,|X^{n}-Y^{n}|^2
+h^2|a(X^{n},\nabla u^{n}_\e(X^{n}))-a(Y^{n},\nabla u^{n}_\e(Y^{n}))|^2\,,
\nonumber
\end{align}
i.e., with the same constant ${\cal A}$ as in Theorem \ref{th:convergence},
$$
|X^{n+1}-Y^{n+1}|^2\le(1+2Ch)|X^n-Y^n|^2+8\,{\cal A}^2\,h^2
$$
and iterating over $n$ : $|X^n(x)-X^n(y)|^2\le e^{2C\,T}\left[|x-y|^2+8\,{\cal A}^2\,T\,h\right]\,$.
\label{rm:explicit scheme}
\end{remark}
%
\section{The fully discrete  semi-lagrangian scheme}
\label{Sec:fully discrete}
In this section we introduce a finite element discretization of \eqref{HJn} and an approximation of~\eqref{mn} by a bounded discrete measure, yielding a fully discrete scheme for \eqref{HJ}-\eqref{CE}-\eqref{IC}.

For an arbitrarily fixed space step $k>0$, we consider the regular uniform grid of $\R^d$ given by $\cX^k:=\{x_i=ik\,,\; i\in\Z^d\}$ . Let $\mathcal{T}^k=\{S^k_j\}_{j\in\J^k}$ be the associated collection of non-degenerate, pairwise disjoint and uniform simplices having as vertices lattice points $x_i\in\cX^k$ and covering $\R^d$, ($S^k_j$ are triangles in dimension 2 and tetrahedra in  dimension 3). We denote also by
$$
W^k=\{w\in C(\R^d) :\text{$w$ is linear on $S_j^k$, $j\in\J^k$}\},
$$
the space of continuous piecewise linear functions on $\mathcal{T}^k$. Then, each $w\in W^k$ can be expressed as
\begin{equation}
w(x)=\sum_{i\in \Z^d}\beta^k_i(x)w(x_i),
\label{eq:intoperator}
\end{equation}
for basis functions $\beta^k_i\in W^k$ satisfying $\beta^k_i(x_j)=\delta_{ij}$ for $i,j\in\Z^d$. It immediately
follows that any $\beta^k_i$ has compact support, $0\leq \beta^k_i(x)\leq 1$, $\sum_{i\in\Z^d}\beta^k_i(x)=1$ and at any $x\in\R^d$ at most $(d+1)$ functions $\beta^k_i$ are non-zero.
 In the sequel, the interpolation operator defined by \eqref{eq:intoperator} will be denoted $P_k$.

The fully discrete approximation of \eqref{HJn} based on the above space discretization is naturally given~by
\begin{equation}
\left\{
\begin{array}{lrc}
u^n_k(x)=\sum_{i\in \Z^d}\beta^k_i(x)\,u^n_{k,i}\,,\quad\quad n=0,\cdots,N\,,\\ \\
u^{n+1}_{k,i}=\inf_{\xi\in\R^d}\left\{u^n_k(x_i-\xi h)+h\,H^*(x_i,t^n,\xi)\right\}\,,\quad i\in\Z^d\,,
\end{array}
\right.
\label{HJnk}
\end{equation}
where $u^0_{k,i}=u_0(x_i)$. It is obvious that, thanks to the continuous piecewise linear interpolation of the discrete approximation  $(u^{n}_{k,i})_{i\in\Z^d}$ on the lattice at any time step, the continuous function $u^n_k\in W^k$ shares the properties of the semi-discrete approximation $u^n$ given in Lemma \ref{lm: u^nproperties}. Therefore, the equivalent of Lemma \ref{lm: u^nproperties} for the fully discrete approximation $u^n_k$ will be skipped here. On the other hand, the semiconcavity of $u^n_k$, given a semiconcave initial data $u_0$, is not straightforward.
Therefore, we shall give the  equivalent of Lemma \ref{lm:semiconcavity} here and the proof in the Appendix.
\begin{lemma}[Semiconcavity]
Let $u_0\in W^{1,\infty}(\R^d)$ be semiconcave. Then, under hypothesis $(\mathbf{H}_1)$, $(\mathbf{H}_2')$, $(\mathbf{H}_3)$, $(\mathbf{H}_4)$ and $(\mathbf{H}_5)$, $u^n_k$ is discretely semiconcave for all $n=0,\cdots N$, i.e.
\begin{equation}
u^n_{k}(x+x_j)-2\,u^n_{k}(x)+u^n_{k}(x-x_j)\le (C_{conc}^{u_0}+T\,C_{conc}^{H^*})|x_j|^2\,,\quad x\in\R^d\,,\;j\in\Z^d\,,
\label{grid_semiconcavity}
\end{equation}
and weakly semiconcave on $\R^d$, i.e.
\begin{equation}
u^n_{k}(x+y)-2\,u^n_{k}(x)+u^n_{k}(x-y)\le (C_{conc}^{u_0}+T\,C_{conc}^{H^*})\left[|y|^2+\f{k^2}2(E(x+y)+E(x-y))\right]
\,,\quad x,y\in\R^d\,,
\label{w_semiconcavity}
\end{equation}
where $E(x)$ is a nonnegative, continuous and bounded function vanishing over $\cX^k$.
\label{lm:u^nksemiconcavity}
\end{lemma}
\begin{remark} We have chosen to consider a uniform grid $\cX^k$, i.e. a uniform space step $k$ in any axial direction, only for simplicity of notations. It is obvious that Lemma \ref{lm:u^nksemiconcavity} still holds true if one chooses different space steps for each direction of the grid, but with an additional non-degeneracy condition. On the other hand, it seems difficult to obtain the discrete-semiconcavity of $u^n_k$ if the regular lattice $\cX^k$ is replaced with a general non-degenerate  triangulation of $\R^d$. Moreover, this key property strongly depends  on the continuous piecewise linear interpolation \eqref{eq:intoperator}. Therefore, it is not possible to use a nonlinear interpolation operator in order to preserve the semiconcavity of $|x|^2$ and to obtain the semiconcavity of $u^n_k$ over $\R^d$ from \eqref{grid_semiconcavity}.
\end{remark}

We now turn to the approximation of \eqref{mn}. A definition of space continuous trajectories is however always necessary. Therefore, let $u^n_{k,\e}=u^n_k*\rho_\e$ and set

\begin{equation}
\left\{
\begin{array}{l}
X^{n+1}_k=X^n_k+h\,a(X^{n+1}_k,\nabla u^{n+1}_{k,\e}(X^{n+1}_k))\,,\quad \quad n=0,\cdots,N-1,\\
X^0_k=x\,,\quad\quad x\in\R^d\,.
\end{array}
\right.
\label{eq:Eulerimplicit2}
\end{equation}
Again, the existence of $X^{n+1}_k$ given $X^n_k$ is due  to the Brouwer fixed point theorem applied to the map $y~\mapsto~X^n_k+h\,a(y,\nabla u^{n+1}_{k,\e}(y))$. However, the same argument as in Lemma \ref{lm:Xh} giving the uniqueness of $X^{n+1}_k$ cannot be reproduced here. Indeed, due to the weak semiconcavity~\eqref{w_semiconcavity}, $u^n_{k,\e}$ is only weak semiconcave in general. Therefore,  we are allowed to assume only the following weak ($\hbox{OSL}_h^k$) condition,
\begin{equation}
\left(a(x,\nabla u^n_{k,\e}(x))-a(y,\nabla u^n_{k,\e}(y))\right)\cdot(x-y)\le C'\,|x-y|^2,\quad x,y\in\R^d\,,\quad |x-y|\ge k\,,
\label{OSLhk}
\end{equation}
for $n=0,\cdots,N$ and a constant $C'$ independent of $h$, $k$ and $\e$.

In order to obtain the well posedeness of the implicit Euler scheme \eqref{eq:Eulerimplicit2}, we assume
\begin{itemize}
\item[$(\mathbf{A}_2)$] $a$ locally Lipschitz w.r.t. $p$ uniformly in $x$ and one sided Lipschitz continuous w.r.t. $x$ locally uniformly in $p$.
\end{itemize}
Then, combining hypothesis $(\mathbf{A}_2)$ with the Lipschitz property of $u^n_k$, there exists a constant $C''$, independent of $h,\,k$ and $\e$, such that if $h$ is small enough, i.e.
\beq
C''\, h\,\e^{-1}<1\,,
\label{eq:hcond}
\eeq
the previous map is contracting,  $X^{n+1}_k$ is unique and the flow $(n,x)\mapsto X^n_k(x)$ is locally uniformly  bounded and Lipschitz continuous w.r.t. $(x,n)\in Q_h$, uniformly in $k$.

\begin{remark}
Condition \eqref{OSLhk} gives us nonetheless the interesting property that the characteristics $X^n_k$ do not move  much away from each other. Indeed, for $X^{n+1}_k$ and $Y^{n+1}_k$ defined by \eqref{eq:Eulerimplicit2} and such that $|X^{n+1}_k-Y^{n+1}_k|\ge k$, proceeding as in \eqref{eq:estequicontinuity}, we obtain
$$
|X^{n+1}_k-Y^{n+1}_k|\le (1+C'\,h\,\delta^{-1})|X^{n}_k-Y^{n}_k|\,.
$$
Therefore,
$$
|X^{n+1}_k-Y^{n+1}_k|\le (1+C'\,h\,\delta^{-1})\max\{|X^{n}_k-Y^{n}_k|,\,k\}\,,
$$
and iterating over $n$,
\begin{equation}
\label{stab1}
|X^{n}_k(x)-X^{n}_k(y)|\le\exp(C'\,T\,\delta^{-1})\max\{|x-y|,\,k\}\,.
\end{equation}
\end{remark}

Next, let $\delta_i$ be the Dirac measure concentrated on the lattice points $x_i$. Let $m^0_k:=\sum_{i\in\Z^d}m^0_{k,i}\,\delta_i$ be an approximation of the initial measure $m_0$ in the space of discrete bounded measures, for the $\cM_1(\R^d)\,w-*$ topology, conserving the mass and that tends to 0 at infinity uniformly w.r.t. $k$ sufficiently small, i.e. for all $\e>0$ there exist $R>0$ and $K>0$ s.t.
\beq
|m^0_k|(\R^d\setminus B_R(0))<\e\qquad\forall\ k<K\,.
\label{eq:compactenessm0k}
\eeq
For example, one can consider $m^0_{k,i}:=m_0(A^k_i)$, with $(A^k_i)_{i\in\Z^d}$ a partition of $\R^d$ such that $x_j\in A^k_i$ iff $i=j$. Then, the mass is obviously conserved and \eqref{eq:compactenessm0k} is satisfied since for any Borel set $B$
\[
|m^0_k|(B)\le|m_0|\left(\cup_{i\in I_k}A^k_i\right)\,,\qquad I_k=\{i\in\Z^d\,:\,x_i\in B\}\,.
\]

We define $\mu^n_k$ as the image of $m^0_k$ by means of the flow $X^n_k(\cdot)$. As for $m^n$, $\mu^n_k$ is well defined in $\cM_1(\R^d)$, the map $x\mapsto X^n_k(x)$ being uniquely defined, continuous and onto from $\R^d$ to $\R^d$. $\mu^n_k$ satisfies also all the properties (i)-(v) with respect to $m^0_k$, for any $n$. However, $\mu^n_k$ is not a discrete bounded measure. Hence, following the classical procedure of the finite element approximation, we observe first that to determine $\mu^n_k$, it is sufficient to test the equation
\beq
\langle \mu^n_k,\phi\rangle= \langle m^0_k,\phi(X^n_k(\cdot))\rangle\,,
\label{eq:munk}
\eeq
against any $\phi\in C^0_c(\R^d)$. Indeed, $C^0_c(\R^d)$ is dense in $C^0_0(\R^d)$ and $ \mu^n_k$ tends to 0 at infinity, uniformly in $t$ and in $k$ sufficiently large owing to \eqref{eq:compactenessm0k}. Moreover, since any function in $C^0_c(\R^d)$ can be approximated uniformly by a function $w\in W^k$ with compact support, \eqref{eq:munk} becomes for such test functions
\beq
\langle \mu^n_k,w\rangle=
\sum_{i\in\Z^d}\,\sum_{j\in\Z^d}m^0_{k,i}\,\beta^k_j(X^n_k(x_i))\,w(x_j)=\sum_{j\in\Z^d}m^n_{k,j}\,w(x_j)\,,
\label{eq:defmnk}
\eeq
with $m^n_{k,j}:=\sum_{i\in\Z^d}m^0_{k,i}\,\beta^k_j(X^n_k(x_i))$. Let us observe that the last identity in \eqref{eq:defmnk} holds since the series $\sum_{i\in\Z^d}m^0_{k,i}$ is absolutely convergent and for the same reason $m^n_{k,j}$ is well defined. Finally,  \eqref{eq:defmnk} leads naturally to the following definition of the discrete bounded measure approximating $\mu^n_k$
\beq
m^n_k:=\sum_{i\in\Z^d}m^n_{k,i}\,\delta_i\,.
\label{eq:defmnk2}
\eeq
It immediately follows that $m^n_k$ conserves the mass, i.e.
$$
\sum_{i\in\Z^d}m^n_{k,i}=\sum_{i\in\Z^d}m^0_{k,i}=m_0(\R^d)\,,\quad\quad n=0,\dots,N\,.
$$
Furthermore, property \eqref{eq:compactenessm0k} is also satisfied by  $m^n_k$ uniformly in $n=0,\dots,N$, by the    definition of the coefficients
$m^n_{k,i}$ and \eqref{stab1}.

With the above definitions, set $U^n_k=(u^n_{k,i})_{i\in\Z^d}$ and $M^n_k=(m^n_{k,i})_{i\in\Z^d}$. The fully discrete scheme \eqref{HJnk}, \eqref{eq:defmnk2}, reads
\begin{equation}
\left\{
\begin{split}
&U^{n+1}_k=\inf_{\xi\in\R^d}\{B^k(\xi)\,U^n_k+h\,L^{n}(\xi)\}\,,\quad \quad n=0,\cdots,N-1\,,\\
&M^{n}_k=\Lambda^n_k\,M^0_k\,,\quad \quad n=1,\cdots,N\,,
\end{split}
\right.
\label{eq:fullydiscrete}
\end{equation}
where
$B^k(\xi)=(\beta_i^k(x_j-h\xi))_{i,j\in\Z^d}$, $ L^{n}(\xi)=(H^*(x_i,t^n,\xi))_{i\in \Z^d}$ and $\Lambda^n_k=(\beta_i^k(X_k^n(x_j)))_{i,j\in\Z^d}$. It is worth noticing that $B^k$ is a stochastic matrix, while $\Lambda^n_k$ is the transpose of a stochastic matrix.

Defining again the following time piecewise constant approximations
\begin{equation}
u_{h,k}(x,t)=u^{[t/h]}_k(x)\,,\quad (x,t)\in Q_T\,, \quad\quad
m_{h,k}^\e(t)=m^{[t/h]}_k\,,\quad t\in[0,T]\,,
\label{eq:uhk_mehk}
\end{equation}
as well as $u_{h,k}^\e=u_{h,k}*\rho_\e$ and the time linear interpolated trajectories $X_{h,k}^\e$ exactly as in \eqref{eq:Xh}, we can prove the convergence of the fully discrete scheme \eqref{eq:fullydiscrete}.

\begin{theorem}[Convergence]
Let $u_0\in(W^{1,\infty}\cap BUC)(\R^d)$ be semiconcave and $m_0\in\cM_1(\R^d)$. Assume hypothesis $(\mathbf{H}_1)$, $(\mathbf{H}_2')$, $(\mathbf{H}_3)$, $(\mathbf{H}_4)$, $(\mathbf{H}_5)$, $(\mathbf{A}_1)$ and $(\mathbf{A}_2)$. Assume in addition that the ($\hbox{OSL}_h^k$) condition \eqref{OSLhk} is satisfied with $h$ sufficiently small such that  $C'\,h<1$ and that \eqref{eq:hcond} is satisfied with $\e=h^\alpha$, $\alpha\in(0,1)$. Then, as $\f k{h^{1+\alpha}}\to0$,
\begin{itemize}
\item[(i)] $u_{h,k}\to u$ in $L^\infty(Q_T)$ and $\nabla u_{h,k}^{\e}(x,t)\to \nabla u(x,t)$ in any point $(x,t)\in Q_T$ of differentiability of~$u$, where $u$ is the unique viscosity solution of~\eqref{HJ};
\item[(ii)] $X_{h,k}^\e$ converges locally uniformly in $Q_T$ toward the unique Filippov characteristic $X$ associated to the vector field $a(\cdot,\nabla u)$;
\item[(iii)] $m_{h,k}^{\e} \rightharpoonup m$ in $C^0([0,T];\cM_1(\R^d)\,w-*)$, where $m$ is the unique measure solution of \eqref{CE}.
\end{itemize}
\label{th:fullyconvergence}
\end{theorem}
\begin{proof}
We shall prove that there exists a constant $C$ independent of $h$ and $k$, such that
\beq
\|u^n_k-u^n\|_{L^\infty(\R^d)}\le C(n+1)k\,,\qquad n=0,\dots,\,N\,.
\label{eq:estunk-un}
\eeq
As a consequence, by the time regularity of $u$ and by estimate \eqref{eq:estun}, it follows that
$$
\|u_{h,k}(t)-u(t)\|_{L^\infty(\R^d)}\le C(h+\f kh+h^{1/2})\,,\qquad \forall\ t\in[0,T]\,.
$$

Estimate \eqref{eq:estunk-un} is obvious for $n=0$, being $u^0_k=P_ku_0$ and $u^0=u_0$. Let us suppose that it holds true for a given $n$. Then, for any argument $\alpha^n(x_i)\in A^n(x_i)$, we have by \eqref{HJn}, \eqref{arginf} and \eqref{HJnk}
$$
u^{n+1}_{k,i}-u^{n+1}(x_i)\le u^n_k(x_i-h\,\alpha^n(x_i))-u^n(x_i-h\,\alpha^n(x_i))\le C(n+1)k\,.
$$
Exchanging the roles of $u^{n+1}_{k,i}$ and $u^{n+1}(x_i)$ we obtain
$$
|u^{n+1}_{k,i}-u^{n+1}(x_i)|\le C(n+1)k\,,\quad\forall\ i\in\Z^d\,,
$$
so that for any $x\in\R^d$,
$$
|u^{n+1}_k(x)-u^{n+1}(x)|\le|u^{n+1}(x)-P_ku^{n+1}(x)|+\sum_{i\in\Z^d}\beta_i^k(x)|u^{n+1}(x_i)-u^{n+1}_{k,i}|\le Ck+C(n+1)k\,.
$$
Since the convergence of $\nabla u_{h,k}^{\e}$ can be proved exactly as in Theorem \ref{th:convergence}, from the weak-semiconcavity of $u^n_{k,\e}$, statement (i) is proved.

Next, from \eqref{eq:estunk-un}, we have the same estimate for the regularized approximation $u^n_{k,\e}$ and $u^n_\e$, yielding
\beq
\|\nabla u^n_{k,\e}-\nabla u^n_\e\|_{L^\infty(\R^d)}\le C\f k{\e\,h}\,,\qquad n=0,\dots, N\,.
\label{eq:estgrad}
\eeq
Then, from the $(\hbox{OSL}_h^k)$ condition \eqref{OSLhk}, the $(\hbox{OSL}_h)$ condition \eqref{OSLh} with constant $C'$ follows as $k\to0$, and the trajectories $X^n$ can be defined as in \eqref{eq:Eulerimplicit}, under the assumption $C'h<1$. Moreover,
\begin{align}
&|X^{n+1}_k-X^{n+1}|^2=(X^{n+1}_k-X^{n+1})\cdot(X^n_k-X^n)
\nonumber\\
&\quad\quad \quad +h\,(X^{n+1}_k-X^{n+1})\cdot\left(a(X^{n+1}_k,\nabla u^{n+1}_{k,\e}(X^{n+1}_k))-a(X^{n+1},\nabla u^{n+1}_\e(X^{n+1}))\right)
\nonumber\\
&\quad\quad\le|X^{n+1}_k-X^{n+1}||X^n_k-X^n|
\nonumber\\
&\quad\quad \quad +h\,(X^{n+1}_k-X^{n+1})\cdot\left(a(X^{n+1}_k,\nabla u^{n+1}_{k,\e}(X^{n+1}_k))-a(X^{n+1}_k,\nabla u^{n+1}_\e(X^{n+1}_k))\right)
\nonumber\\
&\quad\quad \quad +h\,(X^{n+1}_k-X^{n+1})\cdot\left(a(X^{n+1}_k,\nabla u^{n+1}_{\e}(X^{n+1}_k))-a(X^{n+1},\nabla u^{n+1}_\e(X^{n+1}))\right)
\nonumber\\
&\quad\quad\le|X^{n+1}_k-X^{n+1}||X^n_k-X^n|+h\,(\hbox{Lip}_pa)\,|X^{n+1}_k-X^{n+1}|
|\nabla u^{n+1}_{k,\e}(X^{n+1}_k)-\nabla u^{n+1}_{\e}(X^{n+1}_k)|
\nonumber\\
&\quad\quad \quad +C'\,h|X^{n+1}_k-X^{n+1}|^2\,,
\nonumber
\label{eq:estXnk-Xn}
\end{align}
i.e. using \eqref{eq:estgrad},
$$
|X^{n+1}_k-X^{n+1}|\le (1+C'\,h\,\delta^{-1})\left(|X^n_k-X^n|+\tilde C\,\f k{\e}\right)\,,
$$
for $\delta\in(0,1)$ s.t. $C'\,h\le1-\delta$. Iterating over $n$, we obtain for any $x\in\R^d$
\beq
|X^{n}_k(x)-X^n(x)|\le \tilde C\,\f k{\e}\,\sum_{i=1}^n(1+C'\,h\,\delta^{-1})^i\le C\,\f k{\e\,h}\,.
\label{eq:estXnk-Xn2}
\eeq
Combining \eqref{eq:estXnk-Xn2} with statement (ii) in Theorem \ref{th:convergence}, we have proved statement (ii).

Finally, to obtain the convergence of $m_{h,k}^{\e}$ towards $m$, it is enough to prove that, as $k\,,h\to0$,  $(m_{h,k}^\e-m_{h}^\e) \rightharpoonup 0$ in $C^0([0,T];\cM_1(\R^d)\,w-*)$. The latter combined with statement (iii) in Theorem~\ref{th:convergence}, gives us the claim.

Again, using the compactness at infinity of $(m_{h,k}^\e-m_{h}^\e)$, uniformly in $\e\,,h\,,t$ and $k$ sufficiently large,  it is possible to consider only test functions $\phi\in C^0_c(\R^d)$. Then, for $n=\left[\f th\right]$, we~have
\begin{align}
\langle m_{h,k}^\e(t)-m_{h}^\e(t),\phi\rangle&=\langle m^n_k-m^n,\phi\rangle=
\langle m^n_k-m^n,\phi-P_k\phi\rangle+\langle m^n_k-m^n,P_k\phi\rangle
\nonumber\\
&=\sum_{i\in\Z^d}m^n_{k,i}(\phi(x_i)-P_k\phi(x_i))-\langle m_0,\phi(X^n(\cdot))-P_k\phi(X^n(\cdot))\rangle
+\langle m^n_k-m^n,P_k\phi\rangle
\nonumber\\
&:=I_1+I_2+I_3\,.
\nonumber
\end{align}
The term $I_1$ is obviously equal to 0. The term $I_2$ goes to 0 as $k\,,h\to0$ by  Lebesgue dominated convergence theorem. Concerning $I_3$, using \eqref{eq:defmnk}, we have
\begin{align}
I_3&=\langle \mu^n_k,P_k\phi\rangle-\langle m^n,P_k\phi\rangle
=\langle m^0_k,P_k\phi(X^n_k(\cdot))\rangle-\langle m_0,P_k\phi(X^n(\cdot))\rangle
\nonumber\\
&=\langle m^0_k,P_k\phi(X^n_k(\cdot))-P_k\phi(X^n(\cdot))\rangle
+\langle m^0_k-m_0,P_k\phi(X^n(\cdot))\rangle
\nonumber\\
&=\sum_{i\in\Z^d}m^0_{k,i}\left(P_k\phi(X^n_k(x_i))-P_k\phi(X^n(x_i))\right)
+\langle m^0_k-m_0,P_k\phi(X^n(\cdot))\rangle:=I'_3+I''_3\,.
\nonumber
\end{align}
By \eqref{eq:estXnk-Xn2}, the uniform continuity of $P_k\phi$ and the uniform boundedness of $m^0_k$, $I'_3\to0$ as $k\,,h\to0$. The same holds true for $I''_3$ since $m^0_k\rightharpoonup m_0$ in $\cM_1(\R^d)\,w-*$ as $k\to0$.  To conclude, the sequence $\langle m_{h,k}^\e(t)-m_{h}^\e(t),\phi\rangle$ converging to 0, is also equicontinuous on $[0,T]$ since $\langle m_{h}^\e(t),\phi\rangle$ is equicontinuous (see Theorem~\ref{th:convergence}) and
$$
|\langle m_{h,k}^\e(t_1)-m_{h,k}^\e(t_2),\phi\rangle|\le\sum_{i\in\Z^d}|m^{n_1}_{k,i}-m^{n_2}_{k,i}|\,|\phi(x_i)|\,.
$$
Using the uniform continuity of the basis functions, the time equicontinuity of the characteristics $X^n$ and \eqref{eq:estXnk-Xn2} in the estimate below
\begin{align}
|m^{n_1}_{k,i}-m^{n_2}_{k,i}|\le\sum_{j\in\Z^d}|m^0_{k,j}|\left[\right.&
|\beta^k_i(X^{n_1}_k(x_j))-\beta^k_i(X^{n_1}(x_j))|+
|\beta^k_i(X^{n_1}(x_j))-\beta^k_i(X^{n_2}(x_j))|+
\nonumber\\
&|\beta^k_i(X^{n_2}(x_j))-\beta^k_i(X^{n_2}_k(x_j))|\left.\right]\,,
\nonumber
\end{align}
we obtain the claim.
\end{proof}

\section{Appendix}
\label{Sec:Appendix}
\begin{proof} [Proof of Lemma \ref{lm: H^*properties}] (Properties of $H^*$). If $H$ satisfies $(\mathbf{H}_4)$-(i), we need only to prove that $H^*(\cdot,\cdot,0)\in L^\infty(Q_T)$, the other claims being straightforward. Then, let us define $A=\{p\in\R^d\ :\ H(x,t,p)\le M\,,\ \forall\ (x,t)\in Q_T\}$, where $M \equiv\sup_{Q_T}|H(x,t,0)|$. $A$ is a non empty set of $\R^d$ and it is bounded since $H$ growths linearly for $|p|\to+\infty$ and
$$
\f{H(x,t,p)}{|p|}\le \f{M}{|p|}\,,\quad\forall\ p\in A\,.
$$
Thus, there exists $R>0$ such that $A\subset B(0,R)$ and
$$
H^*(x,t,0)=\sup_{p\in\R^d}\{-H(x,t,p)\}=\max_{p\in\overline B(0,R)}\{-H(x,t,p)\}\in L^\infty(Q_T)\,.
$$

On the other end, if $H$ satisfies $(\mathbf{H}_4)$-(ii), the uniform superlinearity of $H^*$ is a direct consequence. To prove~\eqref{eq:H^*}, let us denote $A(\xi)=\{p\in\R^d\ :\ p\cdot\xi-H(x,t,p)\ge-M\,,\ \forall\ (x,t)\in Q_T\}$. Then, $A(\xi)$ is again non empty ($0\in A(\xi)$ for all $\xi$) and bounded whenever $|\xi|\le r$ since
$$
\f{H(x,t,p)}{|p|}\le \f{p\cdot \xi}{|p|}+\f{M}{|p|}\le r+\f{M}{|p|}\,,\quad\forall\ p\in A\,,
$$
uniformly in $(x,t)\in Q_T$. Finally, it follows immediately that $H^*\in L^\infty(Q_T\times \overline B(0,r))$ and
$$
|H^*(x,t,\xi)-H^*(y,t,\xi)|\le \eta(1+R)|x-y|\,,\quad\forall\ x,y\in\R^d\,,\forall\ t\in[0,T]\,,\forall\ \xi\in\overline B(0,r)\,.
$$
\end{proof}
\begin{proof}[Proof of Lemma \ref{lm:u^nksemiconcavity}] (Semiconcavity of $u^n_k$). We follow here \cite{CFN},\cite{FF} and \cite{LT}. The discrete semiconcavity \eqref{grid_semiconcavity} for $x=x_i\in\cX^k$ can be proved by induction. Indeed, this is true for $u^0_k$ by the semiconcavity of $u_0$. Let $\alpha^n_{k,i}$ be one  argument of the infimum in \eqref{HJnk}. Then,
\begin{equation}
u^{n+1}_{k,i+j}-2\,u^{n+1}_{k,i}+u^{n+1}_{k,i-j}\le
u^n_k(x_{i+j}-h\,\alpha^n_{k,i})-2u^n_k(x_{i}-h\,\alpha^n_{k,i})+u^n_k(x_{i-j}-h\,\alpha^n_{k,i})
+h\,C_{conc}^{H^*}|x_j|^2\,.
\label{est:semiconcavity}
\end{equation}
Next, since the point $(x_{i}-h\,\alpha^n_{k,i})$ belongs to a simplex $S_j^k$, it can be written as a convex combination of its vertices, i.e. $(x_{i}-h\,\alpha^n_{k,i})=\sum_l\lambda_lx_l$, where $\lambda_l\in[0,1]$ and $\sum_l\lambda_l=1$. By the regularity of the lattice, the same holds true for $(x_{i\pm j}-h\,\alpha^n_{k,i})=\sum_l\lambda_lx_{l\pm j}$. Finally, being $u^n_k$ piecewise linear on $\mathcal{T}^k$, \eqref{est:semiconcavity} becomes
\[
u^{n+1}_{k,i+j}-2\,u^{n+1}_{k,i}+u^{n+1}_{k,i-j}\le
\sum_l\lambda_lu^n_k(x_{l+j})-2\sum_l\lambda_lu^n_k(x_l)+\sum_l\lambda_lu^n_k(x_{l-j})
+h\,C_{conc}^{H^*}|x_j|^2\,,
\]
and \eqref{grid_semiconcavity} is proved on the lattice $\cX^k$. For an arbitrary $x\in\R^d$, we have again as before $x=\sum_l\lambda_lx_l$ and $x\pm x_j=\sum_l\lambda_lx_{l\pm j}$. Hence,
$$
u^n_{k}(x+x_j)-2\,u^n_{k}(x)+u^n_{k}(x-x_j)=\sum_l\lambda_l[u^{n+1}_{k,l+j}-2\,u^{n+1}_{k,l}+u^{n+1}_{k,l-j}]\,,
$$
and \eqref{grid_semiconcavity} follows.

Set $C=(C_{conc}^{u_0}+h\,C_{conc}^{H^*})$, inequality \eqref{grid_semiconcavity} implies that the continuous piecewise linear interpolation of $(u^{n}_{k,i}-\f C2\,|x_i|^2)_{i\in\Z^d}$ is concave on $\cX^k$ for all $n=0,\cdots N$, and consequently also over $\R^d$. Therefore, for any $x,y\in\R^d$
$$
u^{n+1}_k(x+y)-2u^{n+1}_k(x)+u^{n+1}_k(x-y)\le \f C2\left[
\sum_{i\in \Z^d}\beta^k_i(x+y)|x_i|^2-2\sum_{i\in \Z^d}\beta^k_i(x)|x_i|^2+\sum_{i\in \Z^d}\beta^k_i(x-y)|x_i|^2
\right]\,.
$$
Since the piecewise linear interpolation of $|x|^2$ satisfies
$$
|x|^2\le \sum_{i\in \Z^d}\beta^k_i(x)|x_i|^2\le |x|^2+k^2\,E(x)\,,
$$
with $E(x)$ the approximation error, the weak-semiconcavity \eqref{w_semiconcavity} is also proved.
\end{proof}
%

{\bf Acknowledgment}. This paper has been partially prepared during the visit of F. Camilli at Universit{\'e} Paris-Diderot.
\bibliographystyle{amsplain}

\noindent ${^a}$ Universit{\'e} Paris 7, UFR Math{\'e}matiques,\\
175 rue du Chevaleret, 75013 Paris, France,\\
UMR 7598, Laboratoire Jacques-Louis Lions, F-75005, Paris, France.
 \\
e-mail:achdou@math.jussieu.fr\vskip  1cm


\noindent ${^b}$ Dipartimento di Scienze di Base e Applicate per l'Ingegneria, \\
``Sapienza" Universit{\`a}  di Roma, \\
via Scarpa 16, 0161 Roma (Italy) \\
e-mail:camilli@dmmm.uniroma1.it\vskip  1cm


\noindent ${^c}$ D{\'e}partement de Math{\'e}matiques, \\
Universit{\'e} d'Evry Val d'Essonne, \\
Rue du P{\`e}re Jarlan, F~91025 Evry Cedex \\
email: lucilla.corrias@univ-evry.fr
\end{document}